\documentclass{amsart}
\usepackage{amsmath}
\usepackage{amssymb}
\usepackage{amsthm}
\usepackage{mathrsfs}

\usepackage{graphicx}

\newtheorem{theorem}{Theorem}[section]
\newtheorem{lemma}[theorem]{Lemma}

\theoremstyle{definition}
\newtheorem{definition}[theorem]{Definition}

\newtheorem{corollary}[theorem]{Corollary}

\theoremstyle{remark}
\newtheorem{remark}[theorem]{Remark}

\numberwithin{equation}{section}

\newcommand{\MMM} {\mathcal M}

\begin{document}

\title{More on the rings $B_1(X)$ and $B_1^*(X)$}
\author{Atanu Mondal}
\address{Department of Commerce (E), St. Xavier's college, Mother Teresa sarani,
	Kolkata - 700016, INDIA} 
\email{atanu@sxccal.edu}
\author{A. Deb Ray}

\address{Department of Pure Mathematics, University of Calcutta, 35, Ballygunge Circular Road, Kolkata - 700019, INDIA}
\email{debrayatasi@gmail.com}

\begin{abstract}
This paper focuses mainly on the ring of all bounded Baire one functions on a topological space. The uniform norm topology arises from the $\sup$-norm defined on the collection $B_1^*(X)$ of all bounded Baire one functions. With respect to this topology, $B_1^*(X)$ is a topological ring. It is proved that under uniform norm topology, the set of all units forms an open set and as a consequence of it, every maximal ideal of $B_1^*(X)$ is closed in $B_1^*(X)$ with uniform norm topology. Since the natural extension of uniform norm topology on $B_1(X)$, when $B_1^*(X) \neq B_1(X)$, does not show up these features, a topology called $m_B$-topology is defined on $B_1(X)$ suitably to achieve these results on $B_1(X)$. It is proved that the relative $m_B$ topology coincides with the uniform norm topology on $B_1^*(X)$ if and only if $B_1(X) = B_1^*(X)$. Moreover, $B_1(X)$ with $m_B$-topology is 1st countable if and only if $B_1(X) = B_1^*(X)$. \\
The last part of the paper establishes a correspondence between the ideals of $B_1^*(X)$ and a special class of $Z_B$-filters, called $e_B$-filters on a normal topological space $X$. It is also observed that for normal spaces, the cardinality of the collection of all maximal ideals of $B_1(X)$ and those of $B_1^*(X)$ are the same.
\end{abstract}

\maketitle
\section{Introduction and Preliminaries}

\noindent In \cite{AA}, we have initiated the study of the ring $B_1(X)$ of Baire one functions on a topological space $X$ and have achieved several interesting results. The bounded Baire one functions, denoted by $B_1^*(X)$ is in general a subring of $B_1(X)$. It has been proved in \cite{AA} that for a completely Hausdorff space $X$, total disconnectedness of $X$ is a necessary condition for $B_1^*(X) = B_1(X)$, however, the converse is not always true. It is therefore a question, when do we expect the converse to hold? In this paper, introducing a suitable topology on $B_1(X)$, called $m_B$-topology, we establish a couple of necessary and sufficient conditions for $B_1^*(X) = B_1(X)$. We observe that $(B_1(X), m_B)$ is a 1st countable space if and only if $B_1(X) = B_1^*(X)$, which is an analogue of the result proved in \cite{EH} for pseudocompact spaces.\\\\
\noindent Defining zero sets of Baire one functions in the usual way in \cite{AA2}, we have established the duality between the ideals of $B_1(X)$ with a collection of typical subfamilies of the zero sets of Baire one functions, called $Z_B$-filters on $X$. The study of such correspondence has resemblance to the duality between ideals of $C(X)$, the ring of all real valued continuous functions on $X$ and the $z$-filters on $X$ \cite{GJ}. It is of course a natural query whether such correspondence remains true if we confine the ideals within the class of ideals of $B_1^*(X)$. It is not hard to observe that for any $Z_B$-filter $\mathscr F$ on $X$, $Z_B^{-1}[\mathscr F] \bigcap B_1^*(X)$ is an ideal in $B_1^*(X)$. However, $Z_B$ fails to take an ideal of $B_1^*(X)$ to a $Z_B$-filter on $X$. Introducing a new type of $Z_B$-filter, called $e_B$-filter on $X$, we obtain in this paper a similar correspondence between the ideals of $B_1^*(X)$ and the $e_B$-filters on any normal topological space $X$.  \\\\
\noindent We now state some existing definitions and results that are required for this paper.    \\
The zero set $Z(f)$ of a function $f \in B_1(X)$ is defined by $Z(f) = \{x \in X :f(x) = 0\}$ and the collection of all zero sets in $B_1(X)$ is denoted by $Z(B_1(X))$.\\It is evident that, $Z(f) = Z(|f|)$, for all $f \in B_1(X)$. \\
For any real number $r$, $\mathbf{r} \in B_1(X)$ \big(or $B_1^*(X)\big)$ will always indicate the real valued constant function defined on $X$ whose value is $r$. So $Z(\mathbf{0}) = X$ and $Z(\mathbf{s}) = \emptyset$, for any $s \neq 0$.\\
If a Baire one function $f$ on $X$ is a unit in the ring $B_1(X)$ then $\{x \in X :f (x) = 0\} = \emptyset$. The following results provide some sufficient conditions to determine the units in $B_1(X)$, where $X$ is any topological space..

\begin{theorem} \label{thm1.1}
\cite{AA}	Let $X$ be a topological space and $f \in B_1(X)$ be such that
	$f(x) > 0$ $(or f(x) < 0)$, $\forall x \in X$, then $\frac{1}{f}$ exists and belongs to $B_1(X)$.
\end{theorem}
\noindent Now, if $Z(f)=\{x \in X \ : \ f(x)=0\}= \emptyset$, for some $f \in B_1(X)$,  then $Z(f^2)= \emptyset$ and so by using Theorem \ref{thm1.1}, as $f^2 > 0$ there exists a $g \in B_1(X)$ such that $f(fg)=f^2g=1$, i.e., $f$ is a unit in $B_1(X)$.\\
We must note that this consequence has already been pointed out in \cite{MR}. Therefore, we have the following characterization:
\begin{theorem} \label{thm1.2}
	In any topological space $X$, $f \in B_1(X)$ is a unit in $B_1(X)$ if and only if $Z(f) = \emptyset$.
\end{theorem}

\noindent Clearly, if a bounded Baire one function $f $ on $X$ is bounded away from zero, i.e. there exists some $m \in \mathbb{R}$ such that $0<m\leq f(x), \forall x \in X$ then $\frac{1}{f}$ is also a bounded Baire one function on $X$.
\begin{theorem} \label{thm215}
	\cite{AA} Let $f$ be any Baire one function on $X$ and $g: \mathbb{R \rightarrow \mathbb{R}}$ be a continuous function. Then their composition $g\circ f$ is also a Baire one function.
\end{theorem}
\noindent Here we recall from \cite{LV} that Baire one functions are described in terms of pull-backs of open sets as follows:
\begin{theorem}\cite{LV} \label{P1 thm_4.1} \textnormal{(i)} For any topological space $X$ and any metric space $Y$, $B_1(X,Y)$  $\subseteq \mathscr{F_\sigma}(X,Y)$, where $B_1(X,Y)$ denotes the collection of Baire one functions from $X$ to $Y$ and $\mathscr{F_\sigma}(X,Y)=\{f:X\rightarrow Y : f^{-1}(G)$ is an $F_\sigma$ set, for any open set $G \subseteq Y$\}.\\
	\textnormal{(ii)} For a normal topological space $X$, $B_1(X,\mathbb{R})$  $= \mathscr{F_\sigma}(X,\mathbb{R})$.
\end{theorem}
\begin{theorem}
\cite{AA}	For any $f \in B_1(X)$, $Z(f)$ is a $G_\delta$ set.
\end{theorem}
\begin{definition}
	A nonempty subcollection $\mathscr F$ of $Z(B_1(X))$ is said to be a $Z_B$-filter on $X$, if it satisfies the following conditions:\\
	(1) $\emptyset \notin \mathscr F$\\
	(2) if $Z_1, Z_2 \in \mathscr F$, then $Z_1 \cap Z_2 \in \mathscr F$\\
	(3) if $Z \in \mathscr F$ and $Z' \in  Z(B_1(X))$ is such that $Z \subseteq Z'$, then $Z' \in \mathscr F$.
\end{definition}
\begin{definition}
	A $Z_B$-filter is called $Z_B$-ultrafilter on $X$, if it is not properly contained in any other $Z_B$-filter on $X$.
\end{definition}
\begin{theorem}
	\cite{AA2} If $I$ is any ideal of $B_1(X)$, then $Z_B[I]=\{Z(f): f \in I \}$ is a $Z_B$-filter on $X$.
\end{theorem}
\begin{theorem}
\cite{AA2}	For any $Z_B$-filter $\mathscr F$ on $X$, $Z_B^{-1}[\mathscr F]=\{f\in B_1(X): Z(f) \in \mathscr F\}$ is an ideal in $B_1(X)$.
\end{theorem}
\noindent Let $\mathscr I_B$ be the collection of all ideals in $B_1(X)$ and $\mathscr F_B(X)$ be the collection of all $Z_B$-filters on $X$. The map $Z_B: \mathscr I_B \mapsto \mathscr F_B(X)$ defined by $Z_B(I)=Z_B[I], \forall I \in \mathscr I_B$ is surjective but not injective in general. Although, the restriction map $Z_B: \MMM(B_1(X)) \mapsto \Omega_B(X)$ is a bijection, where $\MMM(B_1(X))$ and $\Omega_B(X)$ are respectively the collection of all maximal ideals in $B_1(X)$ and the collection of all $Z_B$-ultrafilters on $X$ \cite{AA2}.

\noindent A criterion to paste two Baire one functions to obtain a Baire one function on a bigger domain is proved for normal topological spaces in the following theorem:
\begin{theorem}\label{lem0.5}
(\textbf{Pasting Lemma})
    	Let $X$ be a normal topological space with $X=A \cup B$, where both $A$ and $B$ are $G_\delta$ sets in $X$. If $f : A\mapsto \mathbb{R}$ and $g : B\mapsto \mathbb{R}$ are Baire one functions so that $f(x)=g(x)$, for all $x \in A\cap B$, then the map $h: X \mapsto \mathbb{R}$ defined by \\
    	\[ h(x)= \begin{cases} 
    	f(x) &$ if $x \in A$$ \\
    	
    	g(x) &$ if x $\in B$$.

    	\end{cases}
    	\]
    	is also a Baire one function.
    	
    \end{theorem}
\begin{proof}
	Let $C$ be any closed set in $\mathbb{R}$. Then $h^{-1}(C)=f^{-1}(C) \bigcup g^{-1}(C)$. By Theorem \ref{P1 thm_4.1}, $f^{-1}(C)$ and $g^{-1}(C)$ are $G_\delta$ sets in $A$ and $B$ respectively. Since  $A$ and $B$ are $G_\delta$ sets in $X$, $f^{-1}(C)$ and $g^{-1}(C)$ are $G_\delta$ sets in $X$. We know finite union of $G_\delta$ sets is $G_\delta$, so $h^{-1}(C)$ is a $G_\delta$ set in $X$. Therefore, using Theorem \ref{P1 thm_4.1}, we have $h$ is a Baire one function.
\end{proof}
\noindent The result remains true if the $G_\delta$ sets are replaced by $F_\sigma$ sets. The proof being exactly the same as that of Theorem~\ref{lem0.5}, we only state the result as follows:
\begin{theorem}
Let $X$ be a normal topological space with $X=A \cup B$, where both $A$ and $B$ are $F_\sigma$ sets in $X$. If $f : A\mapsto \mathbb{R}$ and $g : B\mapsto \mathbb{R}$ are Baire one functions so that $f(x)=g(x)$, for all $x \in A\cap B$, then the map $h: X \mapsto \mathbb{R}$ defined by \\
    	\[ h(x)= \begin{cases} 
    	f(x) &$ if $x \in A$$ \\
    	
    	g(x) &$ if x $\in B$$.

    	\end{cases}
    	\]
    	is also a Baire one function.
\end{theorem}

\section{$m_B$-topology - An extension of uniform norm topology}
\noindent Let $X$ be any topological space. Defining the `$\sup$'-norm as usual we get that $B_1^*(X)$ is a Banach space. The topology induced by the $\sup$-norm is known as the uniform norm topology and it is not hard to see that $B_1^*(X)$ with uniform norm topology is a topological ring. For any $f \in B_1^*(X)$, the collection $\{B(f, \epsilon) \ : \ \epsilon >0\}$, where $B(f, \epsilon)= \{g \in B_1^*(X) \ : \ ||f-g|| \leq \epsilon\}$, forms a base for neighbourhood system at $f$ in the uniform norm topology on $B_1^*(X)$.
\begin{theorem}\label{units_1}
The set \index{$\mathscr{U}_{B^*}$}$\mathscr{U}_{B^*}$ of all units of $B_1^*(X)$ is an open set in $B_1^*(X)$ with the uniform norm topology. 
\end{theorem}
\begin{proof}
	Let $f \in \mathscr{U}_{B^*}$. So, $f$ is bounded away from $0$, i.e., $\exists$ $\lambda > 0$ such that $f(x) > \lambda$, for all $x \in X$. We observe that for any $g \in B(f, \frac{\lambda}{2})$, $|f(x)-g(x)| \leq \dfrac{\lambda}{2}$, for all $x \in X$. This implies that, $\lambda < |f(x)| \leq |f(x)-g(x)|+|g(x)| \leq \dfrac{\lambda}{2}+|g(x)|$, for all $x \in X$. Therefore, $|g(x)| > \frac{\lambda}{2}$, for all $x \in X$. Hence, $g \in \mathscr{U}_{B^*}$ and $B(f, \frac{\lambda}{2}) \subseteq \mathscr{U}_{B^*}$. This completes the proof that $\mathscr{U}_{B^*}$ is an open set.
\end{proof}

\begin{theorem}\label{closure_ideal}
If $I$ is a proper ideal in $B_1^*(X)$ then $\overline{I}$ (Closure of $I$ in the uniform norm topology) is also a proper ideal in $B_1^*(X)$.
\end{theorem}
\begin{proof}
Since $B_1^*(X)$ is a topological ring, $\overline{I}$ is an ideal. We show that $\overline{I}$ is a proper ideal of $B_1^*(X)$. Since $I$ does not contain any unit, $\mathscr{U}_{B^*} \cap I = \emptyset \implies I \subseteq B_1^*(X) \setminus U_{B^*}$. Now $\mathscr{U}_{B^*}$ being an open set in uniform norm topology, $B_1^*(X) \setminus\mathscr{U}_{B^*}$ is a closed set  containing $I$ which implies $\overline{I} \subseteq  B_1^*(X) \setminus \mathscr{U}_{B^*}$. Therefore, $\overline{I} \cap \mathscr{U}_{B^*} = \emptyset$. Consequently, $1 \notin \overline{I}$, proving $\overline{I}$ a proper ideal of $B_1^*(X)$.
\end{proof}
\begin{corollary}
Each maximal ideal of $B_1^*(X)$ is a closed set in the uniform norm topology.
\end{corollary}

\noindent The natural extension of uniform norm topology on $B_1(X)$ would be the one for which $\{U(f,\epsilon): \epsilon >0\}$ is a base for neighbourhood system at $f\in B_1(X)$, where $U(f, \epsilon)$= $\{g \in B_1(X) \ : \ |f(x)-g(x)|\leq \epsilon, \text{ for all } x \in X\}$. We call this topology the $u_B$-topology on $B_1(X)$. Clearly on $B_1^*(X)$, the subspace topology obtained from the $u_B$-topology coincides with the uniform norm topology. so in that case, $B_1(X)$ is a topological ring and the set of all units in $B_1(X)$ forms an open set in this topology. However, it is neither a topological ring nor the set of all units in $B_1(X)$ forms an open set in this topology, unless $B_1(X) = B_1^*(X)$. 
\begin{theorem}
If $B_1(X) \neq B_1^*(X)$ then $B_1(X)$ is not a topological ring with respect to $u_B$-topology.
\end{theorem}
\begin{proof}
	Let $X$ be a topological space such that $B_1(X) \neq B_1^*(X)$. So, there exists an $f \in B_1(X) \setminus B_1^*(X)$ with $f(x) \geq 1$, $\forall x \in X$. We shall show that the multiplication operation is not continuous at the point $(\mathbf{0}, f)$, for any $f \in B_1(X)$. We choose a neighbourhood $U(\mathbf{0},1)$ of the constant function $\mathbf{0}$. Clearly each function in $U(\mathbf{0},1)$ is bounded on $X$. Let $U(\mathbf{0},\epsilon)$ and $U(f, \delta)$ be arbitrary basic neighbourhoods of $\mathbf{0}$ and $f$ in $B_1(X)$ with $u_B$-topology. We notice that, the constant function $\mathbf{\dfrac{\boldsymbol{\epsilon}}{2}} \in U(\mathbf{0},\epsilon)$ and $f \in U(f, \delta)$ but $\mathbf{\dfrac{\boldsymbol{\epsilon}}{2}}.f$ is an unbounded function. So, $\mathbf{\dfrac{\boldsymbol{\epsilon}}{2}}.f \notin U(\mathbf{0},1)$. This proves that the multiplication is not continuous. Hence, $B_1(X)$ with $u_B$-topology is not a topological ring.
\end{proof}
\noindent In the next theorem we show that the collection of all units of $B_1(X)$ is not an open set in $u_B$-topology, if $B_1(X) \neq B_1^*(X)$.   
\begin{theorem}
If $B_1(X) \neq B_1^*(X)$ then \index{$\mathscr{U}_{B}$}$\mathscr{U}_{B}$, the collection of all units in $B_1(X)$ is not an open set in $u_B$-topology.
\end{theorem}
\begin{proof}
Let $f \in B_1(X) \setminus B_1^*(X)$ be an unbounded Baire one function with $f \geq 1$ on $X$. We define $g = \frac{1}{f}$ on $X$. Since $f(x) > 0$ for all $x \in X$, by Theorem \ref{thm1.1} $g$ is a Baire one function. So $g$ is a positive unit in $B_1(X)$ which takes values arbitrarily close to 0. We show that $g$ is not an interior point of $\mathscr{U}_{B}$. Indeed, for each $\epsilon > 0$, $U(g, \epsilon) \nsubseteq \mathscr{U}_{B}$. We can select a point $a \in X$, such that $0<g(a)< \epsilon$. Taking $h=g-g(a)$ we get that $h \in U(g, \epsilon)$ but $h(a)=0$. So, $h$ is not a unit in $B_1(X)$, i.e., $h \notin \mathscr{U}_{B}.$
\end{proof}
\noindent We put together the outcome of the above discussion in the following theorem: 
\begin{theorem}
For a topological space $X$, the following statements are equivalent:
\begin{enumerate}
	\item[(i)] $B_1(X) = B_1^*(X)$.
	\item[(ii)] $B_1(X)$ with $u_B$-topology is a topological ring.
	\item[(iii)] The set of all units in $B_1(X)$ forms an open set in $u_B$-topology.
\end{enumerate}

\end{theorem}
\begin{corollary}
If a Completely Hausdorff space $X$ is not totally disconnected then $B_1(X)$ with $u_B$-topology is not a topological ring. Moreover, the collection of all units does not form an open set in $B_1(X)$.
\end{corollary}
\begin{proof}
Follows from Theorem 3.4  of \cite{AA}.
\end{proof}
\noindent To overcome the inadequacy of $u_B$-topology, we define another topology on $B_1(X)$ as follows:
Define $M(g,u)= \{f \in B_1(X) \ : \ |f(x)-g(x)| \leq u(x)$, for every $x \in X\}$ and $\widetilde{M}(g,u)= \{f \in B_1(X): |f(x)-g(x)| < u(x)$, for every $x \in X\}$, where $u$ is any positive unit in $B_1(X)$. It is not hard to check that the collection $\mathscr B = \{\widetilde{M}(g,u) \ : \ g \in B_1(X) \text{ and } u \text{ is any positive unit in } B_1(X)\}$ is an open base for some topology on $B_1(X)$.
\begin{theorem}
	$\mathscr{B}$ forms an open base for some topology on $B_1(X)$.	
\end{theorem}
\begin{proof}
	(i) $\mathscr{B}$ covers $B_1(X)$ follows from the construction of $\mathscr{B}$.\\
	(ii) Let $\widetilde{M}(g_1,u_1) \cap \widetilde{M}(g_2,u_2) \neq \emptyset$ and $f \in \widetilde{M}(g_1,u_1) \cap \widetilde{M}(g_2,u_2)$.  Since $f \in \widetilde{M}(g_1,u_1)$, it follows that, for all $x \in X$, $|f(x)-g(x)| < u_1(x)$.\\
	Set $u_1^*(x)=u_1(x)-|f(x)-g_1(x)|$, for all $x \in X$. Then $u_1^* \in B_1(X)$ with $u_1^*(x) >0$, for all $x \in X$. So, $u_1^*$ is a positive unit in $B_1(X)$.\\
	We claim that, $\widetilde{M}(f,u_1^*) \subseteq \widetilde{M}(g_1,u_1)$. \\
	Choose $h \in \widetilde{M}(f,u_1^*)$, which implies that $|h(x)-f(x)| < u_1^*(x), \forall x \in X.$\\
	Now, $|h(x)-g_1(x)| \leq |h(x)-f(x)|+|f(x)-g_1(x)| < u_1^*(x)+u_1(x)-u_1^*(x)=u_1(x).$ So, $h \in \widetilde{M}(g_1,u_1)$.\\
	Analogously, $\widetilde{M}(f,u_2^*) \subseteq \widetilde{M}(g_2,u_2)$, where $u_2^*$ is a positive unit in $B_1(X)$ .\\
	Let $u^*=\min \{u_1^*, u_2^*\}$. Then $u^*$ is a positive unit in $B_1(X)$ and $\widetilde{M}(f,u^*) \subseteq \widetilde{M}(g_1,u_1) \cap \widetilde{M}(g_2,u_2)$. Therefore, $\mathscr{B}$ is an open base for some topology on $B_1(X)$.	
\end{proof}
\begin{definition}
	The topology on $B_1(X)$ for which $\mathscr{B}$ forms a base for open sets is called the \index{$m_B$-topology}$m_B$-topology.
\end{definition}
\begin{remark}
	\textnormal{It is immediate that, for any $g \in B_1(X),$ the collection $\{M(g,u): u$ is a positive unit in $B_1(X)\}$ is a neighbourhood base about $g$ in the $m_B$-topology.}
\end{remark}
\noindent The following theorems show that $B_1(X)$ with $m_B$-topology overcomes the drawbacks of the $u_B$-topology.

\begin{theorem}
$B_1(X)$ with $m_B$-topology is a topological ring.
\end{theorem}
\begin{proof}
We shall show that both addition and multiplication operations are continuous in $m_B$-topology. For any $f, g \in B_1(X)$ we consider a neighbourhood $M(f+g,u)$ at the point $f+g$, where $u$ is a positive unit in $B_1(X)$. Then it is clear that, $M(f, \dfrac{u}{2}) + M(g, \dfrac{u}{2}) \subseteq M(f+g,u)$. So, addition is a continuous map.\\
To show multiplication is continuous at (f, g), we select a positive unit $u$ of $B_1(X)$. We want to produce a positive unit $v$ of $B_1(X)$ with $v \leq 1$ such that $M(f,v).M(g,v) \subseteq M(fg,u)$.	\\
If $v$ has to satisfy the above relation, then we should have, whenever, $|h_1-f| \leq v$ and $|h_2-g| \leq v$ then $|fg-h_1h_2|\leq u$.\\
Now if $|h_1-f| \leq v$ and $|h_2-g| \leq v$, then\\
$|fg-h_1h_2|\\
=|h_1(h_2-g)+g(h_1-f)|\\
=|h_1(h_2-g)+(g-h_2)(h_1-f)+h_2(h_1-f)|\\
\leq |h_1|v+v^2+|h_2|v \\
\leq v[|h_1|+1+|h_2|] \\
\leq v[v+|f|+1+v+|g|] \\
\leq v[|f|+|g|+3]$\\
Therefore, the unit $v$ defined as $v=\bigg(\dfrac{u}{2(|f|+|g|+3)}\bigg) \wedge 1$ will serve our purpose. Hence, $B_1(X)$ with $m_B$-topology is a topological ring.
\end{proof}
\begin{theorem}
$\mathscr{U}_{B}$, the collection of all units in $B_1(X)$ is open in the $m_B$-topology. 
\end{theorem}
\begin{proof}
	Let $u \in \mathscr{U}_{B}$. If we show that $M(u, \frac{|u|}{2}) \subseteq \mathscr{U}_{B}$, then $u$ is an interior point of $\mathscr{U}_{B}$. Indeed, for any $v \in M(u, \frac{|u|}{2}) \implies |v(x)-u(x)| < \frac{|u(x)|}{2}, \forall x \in X \implies v(x) \neq 0$, for all $x \in X$. So by Theorem \ref{thm1.2}, $v \in \mathscr{U}_{B}$. This completes the proof.	
\end{proof}

\begin{theorem}
The closure $\overline{I}$ of a proper ideal $I$ of $B_1(X)$ with $m_B$-topology is also a proper ideal.
\end{theorem}
\begin{proof}
$B_1(X)$ with $m_B$-topology being a topological ring, $\overline{I}$ is an ideal of $B_1(X)$. It is enough to show that $\overline{I}$ is a proper ideal of $B_1(X)$. Since $\mathscr{U}_{B}$ is an open set, proceeding as in Theorem~\ref{closure_ideal} we obtain the result.
\end{proof}

\begin{theorem}
$B_1^*(X)$ is a closed subset of $B_1(X)$ with $m_B$-topology.
\end{theorem}
\begin{proof}
Let $f \in B_1(X) \setminus B_1^*(X)$. Then $M(f,1) \subseteq B_1(X) \setminus B_1^*(X)$ implies that $f$ is not a limit point of $B_1^*(X)$ and this proves the desired result.
\end{proof}
\begin{theorem} \label{5.9}
The uniform norm topology on $B_1^*(X)$ is weaker than the relative $m_B$-topology on $B_1^*(X)$.
\end{theorem}
\begin{proof}
For any $f \in B_1^*(X)$ and $\epsilon > 0$ we get $B(f,\epsilon) =\{g \in B_1^*(X): |f(x)-g(x)| \leq \epsilon\} = M(f, \epsilon) \cap B_1^*(X)$ and this completes the proof.
\end{proof}
\begin{theorem}
Each maximal ideal in $B_1^*(X)$ is closed in $B_1(X)$ with $m_B$-topology.
\end{theorem}
\begin{proof}
Each maximal ideal $M$ in $B_1^*(X)$ is closed in $B_1^*(X)$ with respect to relative $m_B$-topology, since the uniform norm topology on $B_1^*(X)$ is weaker than the relative $m_B$-topology on $B_1^*(X)$ and $M$ is closed under uniform norm topology. Also $B_1^*(X)$ is closed in $B_1(X)$ with $m_B$-topology. Hence $M$ is closed in $B_1(X)$.
\end{proof}
\begin{theorem}
For a normal ($T_4$) topological space $X$, every closed ideal $I$ in $B_1(X)$ with $m_B$-topology is a $Z_B$-ideal.
\end{theorem}
\begin{proof}
Suppose $f,g \in B_1(X)$ with $g \in I$ and $Z(f)=Z(g)$. It is enough to show that $f \in I$. Since $I$ is closed, we show that for any positive unit $u$ in $B_1(X)$, $M(f, u) \bigcap I \ne \emptyset$. Consider the map $h: X \mapsto \mathbb{R}$ given by 

	\[ h(x)= \begin{cases} 
	0 &$ if $|f(x)| \leq u(x)$$ \\
	
	\frac{f(x)-u(x)}{g(x)} &$ if  $f(x) \geq u(x)$$\\
		\frac{f(x)+u(x)}{g(x)} &$ if  $f(x) \leq -u(x)$$

	\end{cases}
	\]
$h$ is well defined. Observe that the collections $\{x \in X : \ |f(x)| \leq u(x)\}, \{x \in X :f(x) \geq u(x)\}$, $\{x \in X :f(x) \leq -u(x)\}$ are $G_\delta$ sets in $X$ and $0, \frac{f(x)-u(x)}{g(x)}, 	\frac{f(x)+u(x)}{g(x)}$ are Baire one functions on their respective domains. So by Theorem \ref{lem0.5},  $h$ is a Baire one function on $X$.\\
For all $x \in X, |h(x)g(x)-f(x)| \leq u(x)$ implies $hg \in M(f,u)$. But $g \in I$ implies $hg \in I$. Therefore $hg \in M(f,u)\bigcap I$ and $f$ is a limit point of $I$. This completes the proof.
\end{proof}
\noindent In Theorem~\ref{5.9} we have seen that the uniform norm topology is weaker than the relative $m_B$-topology on $B_1^*(X)$. The following theorem shows that it is always strictly weaker unless $B_1(X) = B_1^*(X)$. In other words, the equality of uniform norm topology with relative $m_B$-topology on $B_1^*(X)$ characterizes $B_1(X) = B_1^*(X)$.
\begin{theorem} \label{thm2.18}
The uniform norm topology on $B_1^*(X)$ is same as the relative $m_B$-topology on $B_1^*(X)$ if and only if $B_1(X)=B_1^*(X)$.
\end{theorem} 
\begin{proof}
If $B_1(X)=B_1^*(X)$ then for any $g \in B_1(X)$ and any positive unit $u$ in $B_1(X)$ we get $M(g,u) \cap B_1^*(X)=M(g,u)=\{f \in B_1^*(X): |f(x)-g(x)| \leq u(x)$, for every $x \in X\}$. $u$ is a unit in $B_1(X)$ $\big(=B_1^*(X)\big)$ implies that $u$ is bounded away from $0$. So $u(x) \geq \epsilon$, for all $x \in X$ and for some $\epsilon >0$. Hence $g \in B(g, \epsilon)\big(=M(g,\epsilon)\big) \subseteq M(g,u)$. Therefore $M(g,u)$ is a neighbourhood of $g$ in $B_1^*(X)$ in the uniform norm topology. Hence this result along with Theorem \ref{5.9} implies that uniform norm topology on $B_1^*(X)=$ the relative $m_B$-topology on $B_1^*(X)$.\\
For the converse, let $B_1(X) \neq B_1^*(X)$. Then there exists $f \in B_1(X)$ such that $Z(f) = \emptyset$, $f(x)>0$ for all $x \in X$ and $f$ takes arbitrarily small values near $0$. So $f$ is a positive unit in $B_1(X)$. Now, for any two real numbers $r$, $s$, it will never happen that $|\mathbf{r}-\mathbf{s}| \leq f$. This means that $M(\mathbf{r},f) \cap \{\mathbf{t} \ : \ t \in \mathbb{R}\}=\{\mathbf{r}\}$. So the set $\{\mathbf{r} \ : \ r \in \mathbb{R}\}$ is a discrete subspace of $B_1^*(X)$ in the relative $m_B$-topology. From this it follows that, the scalar multiplication operation $\psi: \mathbb{R} \times B_1^*(X) \rightarrow B_1^*(X)$, defined by $\psi (\alpha, g)=\boldsymbol{\alpha}.g$ is not continuous at $(r,\textbf{s})$, where $r,s \in \mathbb{R}$. Thus $B_1^*(X)$ with relative $m_B$-topology is not a topological vector space and hence the relative $m_B$-topology is not same as uniform norm topology on $B_1^*(X)$.
\end{proof}

\begin{theorem}
$B_1(X)$ with $m_B$-topology is first countable if and only if $B_1(X)=B_1^*(X)$. 
\end{theorem}
\begin{proof}
If $B_1(X)=B_1^*(X)$ then by Theorem \ref{thm2.18}, $B_1(X)$ with $m_B$-topology is first countable.\\
Conversely, let $B_1(X) \neq B_1^*(X)$. So there exists $f \in B_1(X) \setminus B_1^*(X)$. Then $g=f^2+1$ is a positive unit in $B_1(X)$. We can find a strictly increasing sequence $\{a_n\}$ of positive real numbers and a countable subset $\{p_n\}$ of $X$ such that $g(p_n)=a_n$, for all $n \in \mathbb{N}$.\\
Consider any countable collection of positive units in $B_1(X)$, say $\{\pi_n\}$.\\
Let $b_n=\frac{1}{2}\min \big[\pi_1(p_n), \pi_2(p_n),...,\pi_n(p_n)\big]$. Then there always exists a real valued continuous function $\sigma : \mathbb{R} \rightarrow \mathbb{R}$ such that $\sigma(x)>0$, for all $x \in \mathbb{R}^+$ and $\sigma(a_n)=b_n^{-1}$, for all $n \in \mathbb{N}.$\\
Define $\psi(x)=\frac{1}{\sigma ( g(x))}, \forall x \in X$. By Theorem \ref{thm1.1} and Theorem \ref{thm215}, $\psi$ is a positive unit in $B_1(X)$ and $\psi(p_n)=\frac{1}{\sigma(g(p_n)) }=b_n \leq \frac{1}{2}\pi_n(p_n)$, for all $n \in \mathbb{N}$.\\
Clearly $\frac{2}{3}\pi_n \in M(0,\pi_n)$ but $\frac{2}{3}\pi_n \notin M(0,\psi)$, because at each $p_n$, $\psi(p_n) \leq \frac{1}{2} \pi_n(p_n) < \frac{2}{3} \pi_n(p_n)$. So the neighbourhood $M(0, \psi)$ at $0$ contains no $M(0,\pi_n), n=1,2,3,..,$ which shows that at the point $0$ in $B_1(X)$ there is no countable neighbourhood base. \\
Hence, $B_1(X)$ with $m_B$-topology is not first countable.
\end{proof}
\begin{corollary}
$B_1(X)$ with $m_B$-topology is metrizable if and only if $B_1(X)=B_1^*(X)$.
\end{corollary}

\section{$e_B$-ideals and $e_B$-filters in $B_1^*(X)$}
\noindent For any $Z_B$-filter $\mathscr F$ on $X$, $Z_B^{-1}[\mathscr F] \bigcap B_1^*(X)$ is an ideal in $B_1^*(X)$. Although, for any ideal $I$ of $B_1^*(X)$, $Z_B[I]$ is not in general a $Z_B$-filter and it is evident from the following example:\\
Consider the ring $B_1^*(\mathbb{N})$ and the ideal $I$ as the collection of all sequences of real numbers converge to $0$. The sequence $\{\frac{1}{n}\} \in I$ but $Z(\{\frac{1}{n}\})=\emptyset$. $\emptyset \in Z_B[I]$ shows that it is not a $Z_B$-filter. \\
In this section, we introduce a special class of $Z_B$-filters, called $e_B$-filters. We locate a class of special ideals of $B_1^*(X)$, called $e_B$-ideals, which behave the same way as the $Z_B$-ideals in $B_1(X)$. The $e_B$-ideals and $e_B$-filters play the pivotal role to establish the desired correspondence.

\noindent For $f \in B_1^*(X)$ and $\epsilon > 0$, we define $E_B^\epsilon(f)= f^{-1}([-\epsilon, \epsilon])=\{x \in X : |f(x)| \leq \epsilon \}$. Every set of this form is a member of $Z(B_1(X))$, as $E_B^\epsilon(f)= Z ((|f|-\epsilon) \vee 0)$. In fact, the converse is also true, i.e., every zero set $Z(h)$, $h \in B^*(X)$ is of the form $E_B^\epsilon(|h| + \epsilon)$. For every non-empty set $I \big (\subseteq  B_1^*(X) \big )$ we define $E_B^\epsilon[I]= \{E_B^\epsilon(f): f \in I\}$ and $E_B(I)= \bigcup\limits_{\epsilon > 0}E_B^\epsilon[I]=\{E_B^\epsilon(f): f \in I, \epsilon >0\}$. For any collection of zero sets $\mathscr F$ , i.e. $\mathscr F \subseteq Z(B_1(X))$, we consider $E_{B}^{\epsilon ^{-1}} [\mathscr F]=\{f \in B_1^*(X): E_B^\epsilon(f) \in \mathscr F \}$ and define $E^{-1}_B(\mathscr F)= \bigcap\limits_{\epsilon > 0} E_{B}^{\epsilon^{-1}}[\mathscr F]=\{f\in B_1^*(X): E_B^\epsilon(f) \in \mathscr F$,  $\forall \epsilon > 0\}$.\\\\
\noindent One may easily see that for any two ideals $I$, $J$ and subcollections $\mathscr F $ and $\mathscr G$ of $Z(B_1(X))$
\begin{enumerate}
\item[(i)] $I \subseteq J \Rightarrow E_B(I) \subseteq E_B(J)$.
\item[(ii)] $\mathscr F  \subseteq \mathscr G  \Rightarrow E_B^{-1}(\mathscr F ) \subseteq E_B^{-1}(\mathscr G)$.
\end{enumerate} 
\noindent We record the following facts in the following couple of theorems:
\begin{theorem} \label{fact1}
For any subset $I$ of $B_1^*(X)$, $I \subseteq E^{-1}_B\big(E_B(I)\big) $, where $E^{-1}_B\big(E_B(I)\big)= \{f \in B_1^*(X): E^\epsilon _{B}(f) \in E_B(I)$, $\forall \epsilon > 0 \}. $
\end{theorem} 
\begin{proof}
Let $f \in I$. Then $E^\epsilon _{B}(f) \in E^\epsilon _{B}[I] $, $\forall \epsilon > 0$. By definition $f \in E^{-1}_B\big(E_B(I)\big) $. Hence, we have the required inclusion.
\end{proof}
\begin{theorem} \label{fact2}
For any subcollection $\mathscr F$ of $Z(B_1(X))$, $E_B \big(E^{-1}_B(\mathscr F)\big) \subseteq \mathscr F$, where $E_B \big(E^{-1}_B(\mathscr F)\big)= \bigcup \limits_{\epsilon > 0} \{E_B^\epsilon(f):E_B^\delta(f) \in \mathscr F$, for all $\delta > 0$\}.
\end{theorem} 
\begin{proof}
The proof follows trivially from the definitions of $E_B$ and $E_B^{-1}$.
\end{proof}

\noindent It is interesting to note that, in Theorem~\ref{fact1}, the inclusion may be a strict one even if we consider $I$ as an ideal in $B_1^*(X)$.\\
For example, consider the ring $B_1^*(\mathbb{N})$ and the function $f(n)= \frac{1}{n}$ in $B_1^*(\mathbb{N})$. Suppose $I=<f^2>$, i.e., the ideal in $B_1^*(\mathbb{N})$ generated by $f^2$. It is quite clear that, $f \notin I$. Now we will show that $f \in E^{-1}_B\big(E_B(I)\big)$ or equivalently, $E^\epsilon _{B}(f) \in E_B(I)$, $\forall \epsilon > 0$. For any $\epsilon > 0$, $E^\epsilon _{B}(f)= \{n \in \mathbb{N}: |f(n)| \leq \epsilon\}=\{n \in \mathbb{N}:f^2(n)\leq \epsilon^2 \}=E^{\epsilon^2} _{B}(f^2) \in E_B(I)$. So,  $I \subsetneqq E^{-1}_B\big(E_B(I)\big) $.\\
\noindent In Theorem~\ref{fact2} too, the inclusion may be proper even when $\mathscr F $ is a $Z_B$-filter on $X$. \\ 
As an example in support of our claim, we consider the ring $B_1^*(\mathbb{R})$ and the $Z_B$-filter $\mathscr F= \{Z \in Z(B_1(\mathbb{R})): 0 \in Z\}$ on $\mathbb{R}$. Now we consider a function $f: \mathbb{R} \mapsto \mathbb{R}$ defined by $f(x)= \frac{|x|}{|x|+1}$. Clearly, $\{0\}=Z(f) \in \mathscr F$. We show that $\{0\} \notin E_B\big(E^{-1}_B(\mathscr F)\big) $. Let $g = f + \epsilon$, for any arbitrary $\epsilon > 0$. It is easy to observe that $E^\epsilon_B{(g)} =\{0\}$ but if we take any positive number $\delta < \epsilon$ then $E^\delta _B(g)= \emptyset$, which does not belong to $\mathscr F.$ Hence, $\{0\} \notin E_B\big(E^{-1}_B(\mathscr F)\big)$.

\begin{definition}
An ideal $I$ in $B_1^*(X)$ is called an $e_B$-ideal if $I = E^{-1}_B\big(E_B(I)\big) $.\\
$I$ is an $e_B$-ideal if and only if for all $\epsilon > 0,$ $E_B^\epsilon(f) \in E_B(I)$ implies $f \in I$.
\end{definition}
\noindent From the definition it is clear that the intersection of $e_B$-ideals is an $e_B$-ideal.
\begin{definition}
A $Z_B$-filter $\mathscr F$ is said to be an $e_B$-filter if $\mathscr F = E_B \big(E^{-1}_B(\mathscr F)\big)$. Equivalently, $\mathscr F$ is an $e_B$-filter if and only if whenever $Z \in \mathscr F$ there exist $\epsilon >0$ and $f \in B_1^*(X )$ such that $Z = E_B^\epsilon (f)$ and $E_B^\delta (f) \in \mathscr F$, $\forall \delta > 0$.
 \end{definition}
 \noindent In what follows, we consider $X$ to be always a normal ($T_4$) topological space. 
    
\begin{theorem} \label{thm1}
	If $I$ is any proper ideal in $B_1^*(X)$ then $E_B(I)$ is an $e_B$-filter.
\end{theorem}
\begin{proof}
	We prove this in two steps. At first we show that $E_B(I)$ is a $Z_B$-filter and then we establish $E_B(I) = E_B \big(E^{-1}_B(E_B(I)\big)$. To show $E_B(I)$ is a $Z_B$-filter we need to check\\
   	a) $\emptyset \notin E_B(I)$.\\
	b) $E_B(I)$ is closed under finite intersection.\\
	c) $E_B(I)$ is closed under superset.\\
	We assert that $\emptyset \notin E_B(I)$. If our assertion is not true then $\emptyset = E_B^\epsilon(f)$, for some $f \in I$ and some $\epsilon > 0$. So $|f(x)| > \epsilon$, $\forall  x \in X$, which implies that $f$ is bounded away from zero and so $f$ is a unit in $B_1^*(X)$ \cite{AA}. This contradicts that $I$ is proper.\\
	For b), Let $E_B^\epsilon(f)$, $E_B^\delta(g) \in E_B(I)$. Then there exist some $f_1, g_1 \in I $ and $\epsilon_1, \delta_1 >0$ such that $E_B^{\epsilon}(f)=E_B^{\epsilon_1}(f_1)$ and $E_B^{\delta}(g)=E_B^{\delta_1}(g_1)$. We can choose $\epsilon_1, \delta_1$ in such a way that $\delta_1 < \epsilon_1$. Then $E_B^{\delta_1^2}(f_1^2+g_1^2) \subseteq E_B^{\epsilon_1} (f_1) \bigcap E_B^{\delta_1} (g_1) = E_B^\epsilon (f) \bigcap E_B^\delta (g) $. Since $I$ is an ideal, $f_1^2+g_1^2 \in I$ and $E_B^{\delta_1^2}(f_1^2+g_1^2) \in E_B(I)$. Therefore, $E_B(I)$ will be closed under finite intersections if we can show (c), i.e., $E_B(I)$ is closed under supersets, which we shall prove now.\\
	Suppose, $E_B^\epsilon (f) \in E_B(I)$ and $Z(f')$ $ ( f' \in B_1^*(X)$) be any member in $Z(B_1(X))$ so that $E_B^\epsilon(f) \subseteq Z(f')$. We shall show that $Z(f') \in E_B(I)$. Since $E_B^\epsilon(f)=E_B^{\epsilon^2}(f^2)$ and $Z(f')=Z(|f'|)$, we can start with $f, f' \geq 0$.\\ Let $P=\{x \in X:|f(x)| \geq \epsilon \}$. \\
	We define a function $g: X \mapsto \mathbb{R}$ by \\
	 	\[ g(x)= \begin{cases} 
	1 &$ if $x \in E_B^\epsilon(f)$$ \\
	
	f'(x)+ \frac{\epsilon}{f(x)} &$ if x $\in P$$.
	\end{cases}
	\]
	We observe that $E_B^\epsilon(f) \bigcap P = \{ x \in X : f(x)=\epsilon\}$ and $\forall x \in E_B^\epsilon(f) \bigcap P$, $f'(x)+ \frac{\epsilon}{f(x)}= 0 + \frac{\epsilon}{\epsilon}=1.$ It is clear that both $E_B^\epsilon(f)$ and $P$ are $G_\delta$ sets and the constant function $\mathbf{1}$ and $f'(x)+ \frac{\epsilon}{f(x)}$ are Baire one functions on $E_B^\epsilon(f)$ and $P$ respectively. Therefore by Theorem~\ref{lem0.5}, $g$ is a Baire one function on $X$, in fact $g \in B_1^*(X)$. \\
Now consider the function\\
$fg: X \mapsto \mathbb{R}$ given by \\
\[ (fg)(x)= \begin{cases} 
f(x) &$ if $x \in E_B^\epsilon(f)$$ \\

(ff')(x)+ \epsilon &$ if x $\in P$$.

\end{cases}
\].\\
Since $I$ is an ideal $fg \in I$ and it is easy to check that $Z(f')=E_B^\epsilon(fg)$. Hence, $Z(f') \in E_B(I)$. Therefore, $E_B(I)$ is a $Z_B$-filter. By Theorem~\ref{fact1}, we get $I \subseteq E^{-1}_B\big(E_B(I)\big)$. Since the map $E_B$ preserves inclusion, we obtain $E_B(I) \subseteq E_B\big( E^{-1}_B\big(E_B(I)\big)\big)$. Also $E_B(I)$ is a $Z_B$-ideal, so by Theorem~\ref{fact2} $E_B\big( E^{-1}_B\big(E_B(I)\big)\big) \subseteq E_B(I) $. Combining these two we have $E_B\big( E^{-1}_B\big(E_B(I)\big)\big) = E_B(I)$. Hence, $E_B(I)$ is a $e_B$-filter.
\end{proof}

\begin{theorem} \label{thm2}
	For any $Z_B$-filter $\mathscr F$ on $X$, $E_B^{-1}(\mathscr F)$ is an $e_B$-ideal in $B_1^*(X)$.
\end{theorem}
\begin{proof}
	 We first show that $E_B^{-1}(\mathscr F)$ is an ideal in $B_1^*(X)$. Let $f,g \in E_B^{-1}(\mathscr F)$. Therefore for any arbitrary $\epsilon > 0$,  $E_B^{\frac{\epsilon}{2}}(f)$, $E_B^{\frac{\epsilon}{2}}(g) \in \mathscr F$. $\mathscr F$ being a $Z_B$-filter on $X$, $E_B^{\frac{\epsilon}{2}}(f) \bigcap E_B^{\frac{\epsilon}{2}}(g) \in \mathscr F$. Also, we know $E_B^{\frac{\epsilon}{2}}(f) \bigcap E_B^{\frac{\epsilon}{2}}(g) \subseteq E_B^{\epsilon} (f+g)$. Hence, $E_B^{\epsilon} (f+g) \in \mathscr F$, or equivalently $f+g \in E_B^{-1} (\mathscr F)$.\\
	 Now consider $f \in E_B^{-1}(\mathscr F)$ and  $h$ be any bounded Baire one function on $X$ with an upper bound $M > 0$ and $\epsilon$ be any arbitrary positive real number. So $|h(x)| \leq M$, for all $x \in X$. For any point $x \in E_B^{\frac{\epsilon}{M}}(f) \implies |f(x)| \leq \frac{\epsilon}{M} \implies |Mf(x)| \leq \epsilon \implies |f(x)h(x)| \leq \epsilon \implies x \in E_B^\epsilon(fh) $. This implies $E_B^{\frac{\epsilon}{M}}(f) \subseteq E_B^\epsilon(fh)$. So $E_B^\epsilon(fh) \in \mathscr F$, for any arbitrary $\epsilon > 0$. Therefore by definition of $E_B^{-1}(\mathscr F)$,  $fh \in E_B^{-1}(\mathscr F)$. Hence $E_B^{-1}(\mathscr F)$ is an ideal in $B_1^*(X)$.\\
	 By Theorem~\ref{fact1}, $E_B^{-1}(\mathscr F) \subseteq E^{-1}_B\big(E_B(E_B^{-1}(\mathscr F))\big) $. Also by Theorem~\ref{fact2}, $E_B(E_B^{-1}(\mathscr F)) \subseteq \mathscr F$. Since $E_B^{-1}$ preserves inclusion, $E^{-1}_B\big(E_B(E_B^{-1}(\mathscr F))\big) \subseteq E_B^{-1}(\mathscr F)$. Hence $E_B^{-1}(\mathscr F) = E^{-1}_B\big(E_B(E_B^{-1}(\mathscr F))\big) $ and so, $E_B^{-1}(\mathscr F)$ is an $e_B$-ideal.
\end{proof}
\begin{corollary}
The correspondence $I \mapsto E_B(I)$ is one-one from the set of all $e_B$-ideals in $B_1^*(X)$ onto the set of all $e_B$-filters on $X$.
\end{corollary}
\begin{theorem}
	If $I$ is an ideal in $B_1^*(X)$ then $E_B^{-1}\big(E_B(I) \big)$ is the smallest $e_B$-ideal containing $I$.
\end{theorem}
\begin{proof}
It follows from Theorem \ref{thm1} and Theorem \ref{thm2} that $E_B^{-1}\big(E_B(I) \big)$ is an $e_B$-ideal. Also from Theorem \ref{fact1} we have $I \subseteq E_B^{-1}\big(E_B(I) \big)$.  If possible let $\mathscr J$ be any $e_B$-ideal containing $I$. So $I \subseteq \mathscr J$ and since $E_B$ and $E_B^{-1}$ preserve inclusion, we can write $E_B^{-1}\big(E_B(I)\big) \subseteq E_B^{-1}\big(E_B(\mathscr J)\big)= \mathscr J$ $\big($since $  \mathscr J$ is an $e_B$-ideal $\big)$. It completes the proof. 
\end{proof}
\begin{corollary} \label{cor0.8}
Every maximal ideal in $B_1^*(X)$ is an $e_B$-ideal.
\end{corollary}
\begin{proof}
Follows immediately from the theorem.
\end{proof}
\begin{corollary}
 Intersection of maximal ideals in $B_1^*(X)$ is an $e_B$-ideal.
\end{corollary}
\begin{proof}
Straightforward.
\end{proof}
\begin{theorem}
	For any $Z_B$-filter $\mathscr F$ on $X$, $E_B\big( E_B^{-1}(\mathscr F)\big)$ is the largest $e_B$-filter contained in $\mathscr F$.
\end{theorem}
\begin{proof}
Theorem~\ref{thm1}, Theorem~\ref{thm2} and Theorem \ref{fact2} show that $E_B\big( E_B^{-1}(\mathscr F)\big)$ is an $e_B$-filter contained in $\mathscr F$. If possible, let $\mathscr E$ be any $e_B$-filter contained in $\mathscr F$. So $\mathscr E \subseteq \mathscr F \implies E_B\big( E_B^{-1}(\mathscr E)\big) \subseteq E_B\big( E_B^{-1}(\mathscr F)\big) \implies \mathscr E \subseteq E_B\big( E_B^{-1}(\mathscr F)\big)$. This completes the proof.
\end{proof}

\begin{lemma} \label{lem0.11}
	Let $I$ and $J$ be two ideals in $B_1^*(X)$ with $J$ be an $e_B$-ideal. Then $I \subseteq J$ if and only if $E_B(I) \subseteq E_B(J)$.
\end{lemma}
\begin{proof}
	$I \subseteq J \implies E_B(I) \subseteq E_B(J)$ follows from the definition of $E_B(I)$.\\
	For the converse, let $f \in I$. To show that $f \in J$. \\
	Suppose $\epsilon > 0$ be any arbitrary positive number. $f \in I \implies E_B^\epsilon(f) \in E_B(I) \implies E_B^\epsilon(f) \in E_B(J) \implies f \in J $ (since $J$ is an $e_B$-ideal ). Therefore, $I \subseteq J$.
\end{proof}
\begin{lemma}
	For any two $Z_B$-filters $\mathscr F_1$ and $\mathscr F_2$ on $X$, with $\mathscr F_1$ be an $e_B$-filter, $\mathscr F_1 \subseteq \mathscr F_2$ if and only if $E_B^{-1}(\mathscr F_1) \subseteq E_B^{-1}(\mathscr F_2)$
\end{lemma}
\begin{proof}
Straightforward.	
\end{proof}
\begin{lemma} \label{lem13}
	Let $\mathscr A$ be any $Z_B$-ultrafilter. If a zero set $Z$ meets every member of $\mathscr A$ then $Z \in \mathscr A$.
\end{lemma}
\begin{proof}
	Consider the collection $\mathscr A \cup \{Z\}$. By hypothesis, it has finite intersection property. So it can be extended to a $Z_B$-filter with $Z$ as one of its member. As this $Z_B$-filter contains a maximal $Z_B$-filter, it must be $\mathscr A$. So $Z \in \mathscr A.$ 
\end{proof}
\begin{theorem} \label{thm0.14}
	Let $\mathscr A$ be any $Z_B$-ultrafilter. A zero set $Z$ in $Z(B_1(X))$ belongs to $\mathscr A$ if and only if $Z$ meets every member of $E_B\big( E_B^{-1}(\mathscr A)\big)$.
\end{theorem}
\begin{proof}
	$\mathscr A$ is a $Z_B$-ultrafilter, so $E_B\big( E_B^{-1}(\mathscr A)\big) \subseteq \mathscr A$. If we assume $Z \in \mathscr A$ then $Z$ meets every member of $\mathscr A$ \cite{AA2} and hence $Z$ meets every member of $E_B\big( E_B^{-1}(\mathscr A)\big)$.\\
Conversely, suppose $Z$ meets every member of $E_B\big( E_B^{-1}(\mathscr A)\big)$. We claim that $Z$ intersects every member of $\mathscr A$. If not, there exists $Z' \in \mathscr A$ for which $Z \bigcap Z' = \emptyset$. Therefore $Z$ and $Z'$ are completely separated in $X$ by $B_1(X)$ \cite{AA} and there exists $f \in B_1^*(X)$ such that $f(Z)=1$ and $f(Z')=0$. Clearly for all $\epsilon > 0,$ $Z' \subseteq Z(f) \subseteq E_B^\epsilon(f)$. Now $Z' \in \mathscr A$ implies that $E_B^\epsilon(f) \in \mathscr A$, for all $\epsilon > 0$. If we choose $\epsilon < \frac{1}{2}$, then $Z \bigcap E_B^\epsilon(f) = \emptyset$, which contradicts that $Z$ meets every member of $E_B\big( E_B^{-1}(\mathscr A)\big)$. Hence, $Z$ meets every member of $\mathscr A$ and therefore $Z \in \mathscr A$.
\end{proof}
\begin{theorem} \label{thm0.15}
If $\mathscr A$ is any $Z_B$-ultrafilter on $X$ then $E_B^{-1}(\mathscr A)$ is a maximal ideal in $B_1^*(X)$. 
\end{theorem}
\begin{proof}
We know from Theorem \ref{thm2}  that $E_B^{-1}(\mathscr A)$ is an ideal in $B_1^*(X)$. Let $M^*$ be a maximal ideal in $B_1^*(X)$ which contains $E_B^{-1}(\mathscr A)$. By using Lemma \ref{lem0.11} we can write 	$E_B\big(E_B^{-1}(\mathscr A)\big) \subseteq E_B(M^*)$. Theorem \ref{thm1} asserts that $E_B(M^*)$ is a $Z_B$-filter on $X$, therefore every member of $E_B(M^*)$ meets every member of $E_B\big(E_B^{-1}(\mathscr A)\big)$. By Theorem \ref{thm0.14} every member of $E_B(M^*)$ belongs to $\mathscr A$. Hence $E_B(M^*) \subseteq \mathscr A$. Every maximal ideal is an $e_B$-ideal, so $M^*= E_B^{-1}\big(E_B(M^*)\big) \subseteq E_B^{-1}(\mathscr A)$. This implies $M^*= E_B^{-1}(\mathscr A)$ which completes the proof.
\end{proof}

\begin{corollary}
	 For any $Z_B$-ultrafilter $\mathscr{A}$ on $X$, $E_B^{-1}(\mathscr A)= E_B^{-1}\big(E_B\big( E_B^{-1}(\mathscr A) \big) \big)$.
\end{corollary}
\begin{definition}
	An $e_B$-filter is called an $e_B$-ultrafilter if it is not contained in any other $e_B$-filter. In other words a maximal $e_B$-filter is called an $e_B$-ultrafilter.
\end{definition}
\noindent Using Zorn's lemma one can establish that, every $e_B$-filter is contained in an $e_B$-ultrafilter.
\begin{theorem} \label{thm0.19}
	If $M^*$ is a maximal ideal in $B_1^*(X)$ and $\mathscr F$ is an $e_B$-ultrafilter on $X$ then \\
	a) $E_B(M^*)$ is an $e_B$-ultrafilter on $X$.\\
	b) $E_B^{-1}(\mathscr F)$ is a maximal ideal in $B_1^*(X)$.
\end{theorem} 
\begin{proof}
	a) Since $M^*$ is a maximal ideal in $B_1^*(X)$ then by Corollary \ref{cor0.8} we have $M^*=E_B^{-1}\big(E_B(M^*)\big)$. Suppose the $e_B$-filter $E_B(M^*)$ is contained in an $e_B$-ultrafilter $\mathscr F'$. So $E_B(M^*) \subseteq \mathscr F' \implies  E_B^{-1}\big(E_B(M^*)\big)=M^* \subseteq E_B^{-1}(\mathscr F') \implies M^* = E_B^{-1}(\mathscr F')$ (since $M^*$ is maximal ideal). Therefore, $E_B(M^*)= E_B\big(E_B^{-1}(\mathscr F')\big) = \mathscr F'$. Hence, $E_B(M^*)$ is an $e_B$-ultrafilter.\\
	b) Let $M^*$ be any maximal extension of the ideal $E_B^{-1}(\mathscr F)$ in $B_1^*(X)$. Now $E_B^{-1}(\mathscr F) \subseteq M^* \implies \mathscr F= E_B\big( E_B^{-1}(\mathscr F)\big) \subseteq E_B(M^*)$. By part (a) we can conclude that $\mathscr F = E_B(M^*)$, which gives us $E_B^{-1}(\mathscr F)=M^*$. Hence we are done.
\end{proof}

\begin{corollary}
	Let $M^*$ be an $e_B$-ideal, then $M^*$ is maximal in $B_1^*(X)$ if and only if $E_B(M^*)$ is an $e_B$-ultrafilter.
\end{corollary}

\begin{corollary}
	An $e_B$-filter $\mathscr F$ is an $e_B$-ultrafilter if and only if $E_B^{-1}(\mathscr F)$ is a maximal ideal in $B_1^*(X)$.
\end{corollary}

\begin{corollary} \label{cor0.22}
	The correspondence $M^* \mapsto E_B(M^*)$ is one-one from the set of all maximal ideals in $B_1^*(X)$ onto the set of all $e_B$-ultrafilters. 
\end{corollary}
\begin{theorem} \label{thm0.23}
	If $\mathscr A$ is a $Z_B$-ultrafilter then it is the unique $Z_B$-ultrafilter containing $E_B\big( E_B^{-1} (\mathscr A)\big)$. In fact, $E_B\big( E_B^{-1} (\mathscr A)\big)$ is the unique $e_B$-ultrafilter contained in $\mathscr A$.
\end{theorem} 
\begin{proof}
	Let $\mathscr A^*$ be a $Z_B$-ultrafilter containing $E_B\big( E_B^{-1} (\mathscr A)\big)$ and $Z \in \mathscr A ^*$. Clearly $Z$ meets every member of $\mathscr A^*$ and so it meets every member of $E_B\big( E_B^{-1} (\mathscr A)\big)$. By Theorem \ref{thm0.14} $Z \in \mathscr A$ and $\mathscr A^* \subseteq \mathscr A$. Since both $\mathscr A$ and $\mathscr A^*$ are $Z_B$-ultrafilters, hence, we have $\mathscr A=\mathscr A^*$. So $\mathscr A$ is unique one containing $E_B\big( E_B^{-1} (\mathscr A)\big)$.\\
	For the second part, $E_B\big( E_B^{-1} (\mathscr A)\big)$ is $e_B$-ultrafilter follows from Theorem \ref{thm0.15} and Theorem \ref{thm0.19}. To prove the uniqueness we suppose $\mathscr E$ be any $e_B$-ultrafilter contained in $\mathscr A$. Then $\mathscr E \subseteq \mathscr A \implies E_B\big(E_B^{-1}(\mathscr E) \big)=\mathscr E \subseteq E_B\big(E_B^{-1}(\mathscr A) \big) \implies \mathscr E=E_B\big(E_B^{-1}(\mathscr A) \big)$. This completes the proof.
\end{proof}
\begin{corollary} \label{cor0.24}
	Every $e_B$-ultrafilter is contained in a unique $Z_B$-ultrafilter.
\end{corollary}
\begin{proof}
	Let $\mathscr{E}$ be an $e_B$-ultrafilter which is contained in two $Z_B$-ultrafilters $\mathscr{F}_1$ and $\mathscr{F}_2$. Therefore, $\mathscr{E}= E_B\big(E_B^{-1}(\mathscr E)) \subseteq E_B\big(E_B^{-1}(\mathscr F_1)) \subseteq \mathscr{F}_1.$ Now by Theorem \ref{thm0.15}, Theorem \ref{thm0.19} (a) and Theorem \ref{thm0.23} $E_B\big(E_B^{-1}(\mathscr F_1))$ is the unique $e_B$-ultrafilter contained in $\mathscr{F}_1$. Hence, $\mathscr{E}=E_B\big(E_B^{-1}(\mathscr F_1))$. Similarly, we have $\mathscr{E}=E_B\big(E_B^{-1}(\mathscr F_2))$. Therefore, $E_B\big(E_B^{-1}(\mathscr F_1))=E_B\big(E_B^{-1}(\mathscr F_2))$ is contained in both $\mathscr{F}_1$ and $\mathscr{F}_2$. Now let $Z \in \mathscr{F}_1$. Then $\mathscr{F}_1$ being a $Z_B$-ultrafilter, $Z$ intersects every member of $E_B\big(E_B^{-1}(\mathscr F_2))$. So by Theorem \ref{thm0.14}, $Z \in \mathscr{F}_2$. This gives us $\mathscr{F}_1 \subseteq \mathscr{F}_2$. By a similar argument, we can show $\mathscr{F}_2 \subseteq \mathscr{F}_1$. Hence, $\mathscr{F}_1 = \mathscr{F}_2$
\end{proof}
\begin{corollary} \label{cor0.25}
	There is a one to one correspondence between the collection of all $Z_B$-ultrafilters on $X$ and the collection of all $e_B$-ultrafilters on $X$. 
\end{corollary}
\begin{proof}
	Let us define a map ${\lambda}$ from the set of all $Z_B$-ultrafilters on $X$ to the set of all $e_B$-ultrafilters on $X$ by $\lambda(\mathscr A)= E_B\big( E_B^{-1} (\mathscr A)\big) $. We shall show that the map is one-one and onto. To show the injectivity, suppose $ \lambda(\mathscr A) = \lambda(\mathscr B)$, which implies $E_B\big( E_B^{-1} (\mathscr A)\big)=E_B\big( E_B^{-1} (\mathscr B)\big)$. By Theorem \ref{thm0.23} we can say $E_B\big( E_B^{-1} (\mathscr A)\big)=E_B\big( E_B^{-1} (\mathscr B)\big)$ is contained in both $\mathscr A$ and $\mathscr B$. Now let $Z \in \mathscr A$. So it intersects every member of $E_B\big( E_B^{-1} (\mathscr B)\big)$ and by Theorem \ref{thm0.14} $Z \in \mathscr B$. This gives us $\mathscr A \subseteq \mathscr B$. By a similar argument we can show that $\mathscr B \subseteq \mathscr A$. Hence, $\mathscr A = \mathscr B$ and $\lambda$ is one-one. Again, let $\mathscr E$ be any $e_B$-ultrafilter. By Corollary \ref{cor0.24} there is a unique $Z_B$-ultrafilter $\mathscr A$ containing $\mathscr E$. Now, $\mathscr{E} \subseteq \mathscr{A} \implies E_B(E^{-1}_B(\mathscr{E})) \subseteq E_B(E^{-1}_B(\mathscr{A})) \implies \mathscr{E} \subseteq E_B(E^{-1}_B(\mathscr{A})) \subseteq \mathscr{A}$. Therefore by using Thorem \ref{thm0.23}, we get $\mathscr E= E_B\big( E_B^{-1} (\mathscr A)\big) $ and this implies $\lambda$ is onto with $\lambda (\mathscr A)= \mathscr E.$  
\end{proof}
\noindent In \cite{AA2} we have shown that there is a bijection between the collection of all maximal ideals in $B_1(X)$ and the collection of all $Z_B$-ultrafilters on $X$. Using Corollary \ref{cor0.22} and Corollary \ref{cor0.25} we therefore obtain the following:
\begin{theorem}
	For any normal ( $T_4$ ) topological space $X$, $\MMM\big(B_1(X) \big)$ and $\MMM\big(B_1^*(X) \big)$ have the same cardinality, where $\MMM\big(B_1(X) \big)$, $\MMM\big(B_1^*(X) \big)$ are the collections of all maximal ideals in $B_1(X)$ and $B_1^*(X)$ respectively.
\end{theorem}
\begin{proof}
	Let us define a map $\psi : \MMM\big(B_1(X) \big) \rightarrow \MMM\big(B_1^*(X) \big)$ by $\psi(M)=E_B^{-1}\big(Z_B[M]\big)$. From \cite{AA2} we know  that $Z_B[M]$ is a $Z_B$-ultrafilter on $X$ and by Theorem \ref{thm0.15} $E_B^{-1}\big(Z_B[M]\big)$ is a maximal ideal in $B_1^*(X)$. We claim that $\psi$ is one-one and onto.\\
Suppose $\psi(M)=\psi (N)$, for some $M,N \in \MMM\big(B_1(X) \big)$. Clearly, $E_B\big(E_B^{-1}\big(Z_B[M]\big)\big) = E_B\big(E_B^{-1}\big(Z_B[N]\big)\big)$ contained in $Z_B[M]$ as well as $Z_B[N]$. Since each member of $Z_B[M]$ intersects every member of $E_B\big(E_B^{-1}\big(Z_B[N]\big)\big)$, so by Theorem \ref{thm0.14} every member of $Z_B[M]$ belongs to $Z_B[N]$. Therefore, $Z_B[M] \subseteq Z_B[N]$. Similarly, $Z_B[N] \subseteq Z_B[M]$. Hence, $Z_B[M] = Z_B[N]$ and by Theorem 2.11 of \cite{AA2} we have $M=N$, i.e., $\psi$ is one-one.\\
To prove that $\psi$ is onto we let $M^* \in \MMM\big(B_1^*(X) \big)$. By Theorem \ref{thm0.19} $E_B(M^*)$ is an $e_B$-ultrafilter. Let $\mathscr A$ be the unique $Z_B$-ultrafilter containing $E_B(M^*)$. $Z_B^{-1}[\mathscr A]$ belongs to $\MMM\big(B_1(X)\big)$ \cite{AA2}. We know $E_B\big( E_B^{-1}(\mathscr A) \big)$ is the unique $e_B$-ultrafilter contained in $\mathscr A$. Therefore, $E_B\big( E_B^{-1}(\mathscr A) \big)= E_B(M^*)$ and $E_B^{-1}(\mathscr A)= M^*$ (since both $E_B^{-1}(\mathscr A)$ and $M^*$ are maximal ideals in $B_1^*(X)$ and maximal ideals are $e_B$-ideals). Consider $M= Z_B^{-1}[\mathscr A]$, clearly $M \in \MMM\big(B_1(X) \big)$. Now $\psi (M)=E_B^{-1}\big( Z_B[M]\big)=E_B^{-1}\big( Z_B[Z_B^{-1}[\mathscr A]]\big)=E_B^{-1}(\mathscr A)=M^*$ (\cite{AA2}), i.e., $Z_B^{-1}[\mathscr A]$ is the preimage of $M^*$. Hence, $\psi$ is onto and it establishes a one to one correspondence between $\MMM\big(B_1(X) \big)$ and $\MMM\big(B_1^*(X) \big)$.
\end{proof}
\noindent The following property characterizes maximal ideals of $B_1^*(X)$. 
\begin{theorem}
	An ideal $M^*$ in $B_1^*(X)$ is a maximal ideal if and only if whenever $f \in B_1^*(X)$ and every $E_B^\epsilon(f)$ intersects every member of $E_B(M^*)$, then $f \in M^*$.
\end{theorem}
\begin{proof}
	Suppose $f \in B_1^*(X)$ and every $E_B^\epsilon(f)$ meets every member of $E_B(M^*)$, where $M^*$ is a maximal ideal in $B_1^*(X)$. We claim $f \in M^*.$ If not, then the ideal $<M^*,f>$ generated by $M^*$ and $f$ must be equal to $B_1^*(X)$. Hence $1=h+fg$, where $g\in B_1^*(X) $ and $h \in M^*$. Let $u$ be an upper bound of $g$ and $\epsilon$ ($0<\epsilon <\frac{1}{2}$) be any pre assigned positive number.\\
	Then we have $\emptyset=E_B^\epsilon(1)=E_B^\epsilon(h+fg) \supseteq E_B^{\frac{\epsilon}{2}}(h) \bigcap E_B^{\frac{\epsilon}{2}}(fg) \supseteq E_B^{\frac{\epsilon}{2}}(h) \bigcap E_B^{\frac{\epsilon}{2u}}(f) $. This is a contradiction to our hypothesis as $E_B^{\frac{\epsilon}{2}}(h) \bigcap E_B^{\frac{\epsilon}{2u}}(f)=\emptyset$. Hence, $f \in M^*$.\\
	Conversely, suppose $M^*$ is any ideal in $B_1^*(X)$ with the given property. Let $M^{**}$ be a maximal ideal containing $M^*$ in $B_1^*(X)$ and $f \in M^{**}$. As $E_B(M^{**})$ is an $e_B$-filter, $E_B^\epsilon(f)$ meets every member of $E_B(M^{**})$, for any $\epsilon > 0$. Since $E_B(M^*) \subseteq E_B(M^{**})$, $E_B^\epsilon(f)$ meets every member of $E_B(M^{*})$ also. By our hypothesis $f \in M^*$. Therefore, $M^*=M^{**}$.
\end{proof}
\noindent One may easily observe the following remarks.
\begin{remark}
Any $e_B$-ideal of $B_1^*(X)$ is closed in $B_1^*(X)$ with uniform norm topology.	
\end{remark}
\begin{proof}
Let $I$ be any $e_B$-ideal in $B_1^*(X)$. By Theorem \ref{closure_ideal}	$\overline{I}$ (Closure of $I$ in the uniform norm topology) is also a proper ideal in $B_1^*(X)$. Clearly, $I \subseteq \overline{I}$. Let us assume $g \in \overline{I}$ and $\epsilon >0$ be any pre assigned positive real number. Therefore, there must exist an $f \in B(g, \frac{\epsilon}{2}) \bigcap I$. Now for all $x \in E^{\frac{\epsilon}{2}} _{B}(f)$, we have $|g(x)| \leq |g(x)-f(x)| +|f(x)| \leq \epsilon$. So, $E^{\frac{\epsilon}{2}} _{B}(f) \subseteq E^{\epsilon}_{B}(g)$. The $Z_B$-filter $E_B(I)$ contains $E^{\frac{\epsilon}{2}} _{B}(f)$, hence it contains $E^{\epsilon}_{B}(g)$, for all $\epsilon >0$. Therefore, $g \in E_B^{-1}\big(E_B(I) \big)=I$ (as $I$ is an $e_B$-ideal) and $\overline{I} \subseteq I$. This proves that, $I = \overline{I}$, i.e., $I$ is closed.
\end{proof}

\begin{remark}
Any $e_B$-ideal of $B_1^*(X)$ is closed in $B_1^*(X)$ with respect to the relative $m_B$-topology. Since $B_1^*(X)$ is closed in $B_1(X)$ with respect to $m_B$-topology, it then follows that every $e_B$-ideal of $B_1^*(X)$ is a closed set in $B_1(X)$ with $m_B$-topology.
\end{remark}

\end{document}